\documentclass[reqno,11pt]{amsart}
\usepackage{amssymb,amsmath,amsthm,amsfonts,color}

\usepackage{mathrsfs,dsfont, comment,mathscinet}
\usepackage{todonotes}

\usepackage{enumerate,esint}

\usepackage[left=2.6cm,right=2.6cm,top=2.8cm,bottom=2.8cm,marginparwidth=7mm,marginparsep=1mm]{geometry}
\usepackage{mathtools}
\usepackage{accents}
\usepackage{graphicx}
\usepackage{bbm}
\usepackage[abs]{overpic}
\usepackage{tikz} 
\usetikzlibrary{patterns}
\usetikzlibrary{decorations.fractals}
\usetikzlibrary{arrows}
\usepackage{xcolor} 
\usepackage{enumitem}
\usepackage{pgf}

\usepackage{epstopdf}
\usepackage[colorlinks=true, pdfstartview=FitV, linkcolor=blue, 
            citecolor=blue, urlcolor=blue]{hyperref}

\usepackage[normalem]{ulem}

\allowdisplaybreaks

\renewcommand{\epsilon}{\varepsilon}

\numberwithin{equation}{section}

\newtheoremstyle{thmlemcorr}{10pt}{10pt}{\itshape}{}{\bfseries}{.}{10pt}{{\thmname{#1}\thmnumber{ #2}\thmnote{ (#3)}}}
\newtheoremstyle{thmlemcorr*}{10pt}{10pt}{\itshape}{}{\bfseries}{.}\newline{{\thmname{#1}\thmnumber{ #2}\thmnote{ (#3)}}}
\newtheoremstyle{defi}{10pt}{10pt}{\itshape}{}{\bfseries}{.}{10pt}{{\thmname{#1}\thmnumber{ #2}\thmnote{ (#3)}}}
\newtheoremstyle{remexample}{10pt}{10pt}{}{}{\bfseries}{.}{10pt}{{\thmname{#1}\thmnumber{ #2}\thmnote{ (#3)}}}
\newtheoremstyle{ass}{10pt}{10pt}{}{}{\bfseries}{.}{10pt}{{\thmname{#1}\thmnumber{ A#2}\thmnote{ (#3)}}}

\theoremstyle{thmlemcorr}
\newtheorem{theorem}{Theorem}
\numberwithin{theorem}{section}
\newtheorem{lemma}[theorem]{Lemma}

\newtheorem{proposition}[theorem]{Proposition}

\newtheorem{conclusion}[theorem]{Conclusion}

\theoremstyle{thmlemcorr*}
\newtheorem{theorem*}{Theorem}
\newtheorem{lemma*}[theorem]{Lemma}
\newtheorem{corollary*}[theorem]{Corollary}
\newtheorem{proposition*}[theorem]{Proposition}
\newtheorem{problem*}[theorem]{Problem}
\newtheorem{conjecture*}[theorem]{Conjecture}

\theoremstyle{defi}

\theoremstyle{remexample}
\newtheorem{remark}[theorem]{Remark}
\newtheorem{example}[theorem]{Example}

\theoremstyle{ass}

\newcommand{\Acal}{\mathcal{A}}

\newcommand{\Bcal}{\mathcal{B}}

\newcommand{\Dcal}{\mathcal{D}}
\newcommand{\Ecal}{\mathcal{E}}

\newcommand{\Gcal}{\mathcal{G}}

\newcommand{\Mcal}{\mathcal{M}}

\newcommand{\Qcal}{\mathcal{Q}}

\newcommand{\Scal}{\mathcal{S}}
\newcommand{\Tcal}{\mathcal{T}}

\newcommand{\Ycal}{\mathcal{Y}}
\newcommand{\Zcal}{\mathcal{Z}}

\newcommand{\Ibb}{\mathbb{I}}

\DeclareMathOperator{\id}{id}

\DeclareMathOperator{\curl}{curl}

\newcommand{\norm}[1]{\|#1\|}

\newcommand{\absB}[1]{\Bigl|#1\Bigr|}

\newcommand{\dd}{\;\mathrm{d}}

\newcommand{\N}{\mathbb{N}}
\newcommand{\R}{\mathbb{R}}

\newcommand{\weakly}{\rightharpoonup}
\newcommand{\weaklystar}{\overset{*}\rightharpoonup}

\newcommand{\eps}{\epsilon}

\newcommand{\tr}{{\rm Tr\,}}

\newcommand{\Yrig}{Y_{\rm rig}}
\newcommand{\Ysoft}{Y_{\rm soft}}
\newcommand{\CCC}{\color{black}}

\newcommand{\BBB}{\color{black}}

\newcommand{\ups}{^{(s)}}
\newcommand{\upe}{^{(e_1)}}
\newcommand{\upee}{^{(e_2)}}

\DeclareMathOperator{\SO}{SO}

\newcommand{\be}[1]{\begin{equation}\label{#1}}
\newcommand{\een}{\end{equation}}
\newcommand{\ep}{\varepsilon}

\title[Homogenization in crystal plasticity for stratified composites]{On static and evolutionary homogenization in crystal plasticity for stratified composites}

\author{Elisa Davoli}
\address{Institute of Analysis and Scientific Computing, TU Wien, Wiedner Hauptstra\ss e 8--10, 1040 Vienna, Austria}
\email{elisa.davoli@tuwien.ac.at}

\author{Carolin Kreisbeck}
\address{Mathematisch-Geographische Fakult\"at, Katholische Universit\"at Eichst\"att-Ingolstadt, Osten\-stra{\ss}e 28, 85072 Eichst\"att, Germany}
\email{carolin.kreisbeck@ku.de}

\begin{document}

 
\maketitle
\begin{abstract}  
\vspace{-12pt}   
The starting point for this work is a static macroscopic model for a high-contrast layered material in single-slip finite crystal plasticity, identified in~[Christowiak \& Kreisbeck, \textit{Calc.\,Var.\,PDE} (2017)] as a homogenization limit via $\Gamma$-convergence. 
First, we analyze the minimizers of this limit model, addressing the question of uniqueness and deriving necessary conditions. In particular, it turns out that at least one of the defining quantities of an energetically optimal deformation, namely the rotation and the shear variable, is uniquely determined, and we identify conditions that give rise to a trivial material response in the sense of rigid-body motions. 
The second part is concerned with extending the static homogenization to an evolutionary $\Gamma$-convergence-type result  for rate-independent systems in specific scenarios, that is, under certain assumptions on the slip systems and suitable regularizations of the energies,  where energetic and dissipative effects decouple in the limit. Interestingly, when the slip direction is aligned with the layered microstructure, the limiting system is purely energetic, which can be interpreted as a loss of dissipation through homogenization. 
                    
 \vspace{8pt}


\noindent\textsc{Date:} \today.

 \end{abstract}   
\section{Introduction}
Motivated by new trends in technology that require materials with non-standard properties, the study of artificially engineered composites (metamaterials) has been the subject of an intense research activity at the triple point between mathematics, physics, and materials science.

Here, we investigate the effective deformation behavior of a special class of mechanical metamaterials exhibiting the following two features whose interplay generates a highly anisotropic material response: (i) the geometry of the heterogeneities is characterized by periodically alternating layers of two different components; (ii) the material properties of the two components show strong differences, in the sense that one is rigid, while the other one is softer, allowing for large-strain elastoplastic deformations along prescribed slip directions. 

The asymptotic analysis of variational models for such stratified materials with fully rigid components and complete adhesion between the phases was initiated by {\sc Christowiak} and {\sc Kreisbeck} 
in \cite{ChK17}, which is the starting point for this paper.
 More precisely, the subject of \cite{ChK17} is a two-dimensional homogenization problem in the context of finite elastoplasticity, with geometrically nonlinear but rigid elasticity, where the softer component can be deformed along a single active plastic slip system with linear self-hardening. At the core of the homogenization result via $\Gamma$-convergence~\cite[Theorem~1.1]{ChK17} lies the characterization of the weak closures of the set of admissible deformations via an asymptotic rigidity result. 
 In \cite{davoli.kreisbeck.ferreira}, these techniques have been carried forward to a model for plastic composites without linear hardening in the spirit of~\cite{CoT05}, which leads to a variational limit problem on the space of functions of bounded variation. 
Natural generalizations of these models to three (and higher) dimensions, where the material heterogeneities are either layers or fibers are studied~\cite{ChK18} and~\cite{EKR21}, respectively. Note that these two references, which are formulated in context of nonlinear elasticity, use energy densities with $p$-growth for $1<p<+\infty$, and consider also non-trivial elastic energies on the stronger components, which allows treating also very stiff (but not necessarily rigid) reinforcements. 
As for their mathematical structure, all the aforementioned papers feature energies of integral form, characterized by linear or superlinear growth, subject to non-convex differential constraints. 
For related work on alternative approaches to layered and fiber-reinforced high-contrast composites with different choices of scaling relations between the elastic constants, thickness and adhesive parameters, see e.g.~\cite{BoB02, BrE07, Jar13, PaS16}. 

Our goal in this work is twofold. First, we provide an analysis of minimizers for the effective energy functional derived in \cite{ChK17} as a homogenized $\Gamma$-limit; 
 in particular, we address the question of uniqueness and identify necessary conditions for minimizers. Second, for some specific case studies, we complement the static homogenization of~\cite{ChK17} by an evolutionary $\Gamma$-convergence analysis. 
 Generally speaking, evolutionary $\Gamma$-convergence aims at transferring the concept of limit passages in parameter-dependent stationary variational problems to time-dependent settings. For energetic rate-independent systems such a theory was developed by {\sc Mielke}, {\sc Roub\'i\v{c}ek} and {\sc Stefanelli} in~\cite{MRS08}. In parallel to the main feature of $\Gamma$-convergence (see e.g.~\cite{Bra05, Dal93}), which guarantees the convergence of solutions, i.e.,~(almost) minimizers of the parameter-dependent functionals converge to minimizers of the limit functional, evolutionary $\Gamma$-convergence for rate-independent systems implies that energetic solutions, again converge to energetic solutions of the limit system. In cases where energetic solutions do not exist, which is the situation in this paper, one can work instead with solutions to associated approximate incremental problems. For a comprehensive introduction to the topic, we refer to~\cite[Sections~2.3-2.5, Section 3.5.4]{MiR15}; for applications in linearized elastoplasticity, see e.g.~\cite{Han11, MiT07} on homogenization, or~\cite{MiS13} on a rigorous justification through a rigorous linearization of finite-strain plasticity.

  In order to describe our results in more detail, some notation needs to be introduced.
Let $\Omega\subset \R^2$ be a bounded Lipschitz domain, which is assumed to represent the reference configuration of a high-contrast material with bilayered microstructure encoded by the alternation of two horizontal layers, a soft and a rigid one, see Figure \ref{fig:material}. Without loss of generality, we can assume that 
\begin{align}\label{barycenter_Omega}
\int_\Omega x\dd{x}=0,
\end{align} namely that the barycenter of $\Omega$ lies in the origin.
To mathematically describe the geometry of the heterogeneities, consider the periodicity cell $Y:=[0,1)^2$,  which we  subdivide into 
$Y=Y_{\rm soft}\cup Y_{\rm rig}$ with $Y_{\rm soft}:=[0,1)\times [0, \lambda)$ for $\lambda\in (0,1)$ and $Y_{\rm rig}:=Y\setminus Y_{\rm soft}$. All sets are extended by periodicity to  $\R^2$.  The (small) parameter $\ep>0$ describes the thickness of a pair (one rigid, one softer) of fine layers, and can be viewed as the intrinsic periodicity scale of the microstructure. The  collection of all rigid and soft layers in $\Omega$ corresponds to the sets $\ep Y_{\rm rig}\cap\Omega$ and $\ep Y_{\rm soft}\cap\Omega$, respectively.  
%
%
\begin{center}
\begin{figure}[t]
\begin{tikzpicture}[scale=1.5]
		\begin{scope}[scale=0.65] 
			\draw (0,0) to [out=90,in=180] (2,2) to [out=0,in=200] (3,2.25) to [out=20,in=225] (4,3) to [out=45, in=135] (6,3) to [out=-45, in=90] (6.3,2) to [out=-90,in=45] (5,0) to [out=225,in=0] (2,-1) to [out=180,in=-90] (0,0);
		\end{scope}
  \draw(.8,-0.1) node      
        { \begin{normalsize}$\textcolor{black}{\Omega \subset
\R^2}$\end{normalsize}};       
               \draw[black] (2,0.5) -- (2.6,0.5) -- (2.6, 0.75) --
        (2, 0.75) -- (2,0.5);
       \draw[black
        ] (.9+4,-3+4.2) -- (3.5+4,-3+4.2) -- (3.5+4,
-4.25+4.2) --         (.9+4, -4.25+4.2) -- (.9+4,-3+4.2);
        \draw[thick,gray!70,
         dotted] (.9+4,-3+4.2) -- (2,0.75);
        \draw[thick,gray!70,
         dotted] (2.,0.5) -- (.9+4,-4.25+4.2);
         \foreach \y in {-4.25+4.2, -4+4.2, -3.75+4.2, -3.5+4.2,
-3.25+4.2} {\fill[fill=blue!30!white] (0.9+4,\y+.1) -- (3.5+4,\y+.1)
-- (3.5+4,\y+.25) -- (.9+4,\y+.25) -- cycle;};
 \draw[black, |-|] (3.6+4,-4.25+4.7) --      (3.6+4,-4+4.7);
  \draw(3.7+4.02,-4.13+4.7) node      
        {\begin{normalsize}$\textcolor{black}{\ \varepsilon}$\end{normalsize}}; 
        \end{tikzpicture}
\caption{Illustration of the reference configuration $\Omega$ and the stratified microstructure at length-scale $\eps$}
\label{fig:material}
\end{figure}
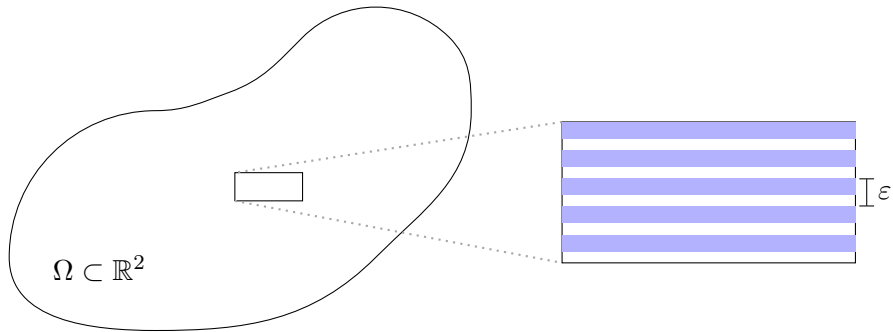
\end{center}
Regarding the material properties, the body as a whole 
 exhibits an elastoplastic behavior characterized by finite single-slip crystal plasticity with rigid elasticity throughout, but the deformations on the individual rigid layers are restricted to global rotations or translations. 
We point out that several different models of finite elastoplasticity have been proposed and analyzed in the literature, see e.g.~\cite{hill} for a general introduction; recent contributions include an analysis of the incompatibility tensor~\cite{amstutz.vangoethem}, a formulation keeping track of frame-invariance of intermediate configurations~\cite{grandi.stefanelli2}, as well as discussions of different (multiplicative) decompositions of deformation gradients~\cite{davoli.francfort, delpiero, RDOC18}.
Here, we adopt the classical approach introduced  in \cite{Kro60, lee}: The gradient of every deformation $u:\Omega\to \mathbb{R}^2$ decomposes into the product of an elastic strain $F_{\rm el}$, and a plastic one, $F_{\rm pl}$, satisfying
\begin{align}\label{multi_decomposition}
\nabla u=F_{\rm el}F_{\rm pl}.
\end{align}
Due to the assumption that the elastic behavior of the body is purely rigid, one has
\begin{align*}
\text{$F_{\rm el}\in SO(2)$\quad almost everywhere (a.e.) in $\Omega$,}
\end{align*} 
where $SO(2)$ denotes the set of rotations in $\R^2$. As for the plastic part, the presence of  a single active slip, with slip direction $s \in \Scal^1:=\{s \in \R^2: |s|=1\}$
and associated slip plane normal $m=s^{\perp}$, translates into
\begin{align}\label{Fpl}
F_{\rm pl}={\rm Id}+\gamma s\otimes m,
\end{align}
where the shear coefficient $\gamma$ measures the amount of slip. While the material is free to glide along the slip system in the softer phase, it is required that $\gamma$ vanishes on the layers consisting of the rigid material, i.e.,~$\gamma=0$ in $\ep Y_{\rm rig}\cap \Omega$. 
 
Collecting these modeling assumptions, we define, for $\ep>0$, the class $\mathcal{A}^{(s)}_\ep$ of admissible
 deformations as 
 \begin{align}\label{eq:def-Aep}
 \begin{split}
\mathcal{A}^{(s)}_\ep& := 
\{u\in \Ycal: \nabla u = R_u({\rm Id}+\gamma_u s\otimes m)
\text{ a.e.~in $\Omega$,} \\ &\hspace{2.5cm}  R_u\in L^\infty(\Omega;SO(2)), \gamma_u\in L^2(\Omega), \text{$\gamma_u=0$ a.e.~in $\ep Y_{\rm rig}\cap
\Omega$} \},
\end{split}
 \end{align}
where $\Ycal =W^{1,2}(\Omega;\R^2)\cap L^2_0(\Omega;\R^2)$, and $L^2_0(\Omega;\R^2)=\{u\in L^2(\Omega;\R^2): \int_\Omega u\dd x=0\}$ denotes the space of $L^2$-functions with zero average in $\Omega$. 
In light of the completely rigid behavior of the body in $\ep Y_{\rm rig}\cap \Omega$, and the plastic deformation behavior along a single slip system, for which we assume linear hardening, in $\ep Y_{\rm soft}\cap \Omega$, the stored elastoplastic energy of a deformation $u\in \Ycal$ reads as
\begin{align}
E_\ep^{(s)}(u) 
&=\begin{cases}
\displaystyle \int_{\ep Y_{\rm soft}\cap \Omega} |\gamma_u|^2\dd x & \text{if $u\in \mathcal{A}^{(s)}_\ep$,}
\\
+\infty & \text{otherwise;} 
\end{cases}\label{eq:defEeps}
\end{align}          
notice that a model for homogeneous materials with the properties of the softer component above has been studied in~\cite{conti13car, conti.dolzmann.kreisbeck}, and we refer the reader to~\cite{CoT05} for a corresponding analysis in the absence of hardening. 

In order to state the homogenization result from~\cite{ChK17}, let us introduce, for any slip direction $s=(s_1, s_2)\in \Scal^1$ 
 the following class of deformations
\begin{align}
\label{eq:def-A}\Acal\ups:=\{u\in \Ycal: \nabla u= R_u(\Ibb + \gamma_u e_1\otimes e_2), R_u\in SO(2), \gamma_u \in L^2(\Omega), \gamma_u\in K\ups \text{ a.e.~in $\Omega$} \},
\end{align}
with
\begin{align}
\label{eq:def-K}K\upee:= \{0\}, \quad K\upe:=\R, \quad \text{and}\quad  
K\ups:=\begin{cases}
[0, -2\lambda \frac{s_1}{s_2}] & \text{if $s_1s_2<0$,}\\ [-2\lambda \frac{s_1}{s_2}, 0] & \text{if $s_1s_2>0$,} 
\end{cases} \quad \text{for $s\notin \{e_1, e_2\}$.}
\end{align}
Throughout the paper, we will often use - without further mention - one of the following alternative representations of the set~in~\eqref{eq:def-A}, that is,
\begin{align}
\nonumber\Acal\ups&=\{u\in \Ycal: u(x) = R_u(x+\Gamma_u(x) e_1)\, \text{for $x\in \Omega$}, \Gamma_u\in W^{1,2}(\Omega)\cap L_0^2(\Omega),\\
&\qquad\qquad\qquad\qquad\label{eq:A-alt1} R_u\in \SO(2), \partial_1 \Gamma_u=0, \partial_2 \Gamma_u\in K\ups\text{ a.e. in }\Omega\}\\
&=\{u\in \Ycal: \nabla u \in \Mcal^{(s)}, \gamma_u\in K\ups \text{ a.e.~in $\Omega$} \},\label{eq:A-alt2}
\end{align}
where\begin{equation}
\label{eq:def-Ms}
\Mcal\ups:=\{F\in \R^{2\times 2}: \det F=1, |Fs|=1\}.
\end{equation} 
We observe that if $u\in \Acal^{(s)}$, then $\curl \nabla u=0$ implies $\partial_1 \gamma_u=0$. Hence, the set $\Acal^{(s)}$  is intrinsically one-dimensional.  

With this notation in place, we can now formulate~\cite[Theorem~1.1]{ChK17}, which characterizes the macroscopic material response of the considered stratified high-contrast composites in terms of $\Gamma$-convergence: 
The $\Gamma$-limit for $\eps\to 0$ of the energies 
$(E_\ep^{(s)})_\eps$ in the weak $W^{1,2}$-topology is given for $u\in \Ycal$ by the functional 
\begin{align}\label{Eups}
E\ups(u) & := \begin{cases}\displaystyle \frac{{s_1}^2}{\lambda} \int_{\Omega}  \gamma_u^2\dd{x} -2s_1s_2\int_{\Omega} \gamma_u \dd{x} & \text{for $u\in \Acal\ups$,}\\
+\infty & \text{otherwise}. \end{cases}
\end{align}
As $\Gamma$-convergence is invariant under continuous perturbations, the previous result is not affected by adding external loading to the stored energy functionals $E_\eps\ups$ in~\eqref{eq:defEeps}.
If we augment $E\ups$ with a term describing work due to a body force with density $g\in L^2(\Omega;\R^2)$, the functional to study is
\begin{align}\label{eq:en-static}
  I\ups_g (u):=E\ups(u) -\int_{\Omega} g\cdot u\dd{x} 
\end{align}
for $u\in \Ycal$. 

 The essence of our first main result can then be summarized in simple terms as follows:
\begin{conclusion}[{Uniqueness of rotations and shear coefficients}] 
\label{thm:main1} For any slip direction $s\in \Scal^1$ and any applied load $g\in L^2(\Omega;\R^2)$, either the rotation or the shear coefficient associated to minimizers of $I\ups_g$ are uniquely determined.
\end{conclusion}
We refer to Section \ref{sec:min-static} for the precise assumptions, as well as to Lemma~\ref{lem:representation} and Proposition~\ref{prop:uniqueness} for the formulation of the statement and the proof of this result. Besides this general observation about uniqueness, we also discuss necessary and sufficient conditions for minimizers under specific conditions on the slip directions and the applied loads, including criteria for trivial deformation behavior in the form of rigid-body motions.  

In the second part of this paper (see Section~\ref{sec:evol}), we expand the homogenization of the introduced static model to a quasistatic context by incorporating time-dependent loadings and dissipation acting on the shear variable; to be more precise, we work with dissipations $\Dcal$ 
 that are given as the difference of the shear coefficients measured in the $L^1$-norm. Our second main contribution regards
the asymptotic analysis of such extended models in the framework of evolutionary $\Gamma$-convergence for energetic rate-independent systems (see~\cite{MiR15}, as well as the beginning of Section~3, where we give a brief outline of this theory adapted to the setting of this paper). 
An intuitive, rough version of our findings reads: 
\begin{conclusion}[{Evolutionary \boldmath{$\Gamma$}-convergence}]
\label{thm:main2}
Suppose that the slip direction is aligned or orthogonal to the orientation of the material microstructure, meaning $s=e_1$ or $s=e_2$. Considering a family of rate-independent systems with (suitably regularized versions of) the stored energies $E_\ep\ups$ from~\eqref{eq:defEeps} and dissipation distance $\Dcal$, 
it follows that energetic and dissipative effects decouple in the limit $\eps\to 0$. Moreover, (approximate) solutions (to the associated time-discrete incremental problems) of the $\ep$-dependent systems converge to energetic solutions of a system involving the stored energy $E\ups$ from~\eqref{Eups} and $\Dcal$. 

If $s=e_1$, the limiting system is very restrictive and corresponds to a purely energetic evolution without any dissipation, so that one can speak of a loss of dissipation through homogenization. 
\end{conclusion}

The precise formulation of Conclusion~\ref{thm:main2} for $s=e_2$ and $s=e_1$ can be found in Theorems~\ref{theo:evol2} and~\ref{theo:evol1}, respectively. 
In both cases, the proof strategy relies on the well-established scheme in~\cite[Section~2.3-2.5]{MiR15}.
The only delicate point is the construction of a so-called ``mutual recovery sequence", for which we utilize tailored arguments that keep track of the special geometry of the problem.

This article is organized as follows: Section \ref{sec:min-static} is entirely devoted to the study of the static minimization problem, while Section~\ref{sec:evol} deals with evolutionary $\Gamma$-convergence.

\subsection*{Notation} 
Throughout the manuscript, $|\cdot|$ denotes the Euclidean norm in $\R^2$, and $\Scal^1$ is the unit sphere in $\R^2$, i.e., $\Scal^1:=\{s\in \R^2:\,|s|=1\}$. 
For $a\in \R^2$ with $a\neq 0$, we use the short-hand notation $\bar{a}=a/|a|$. 
We take $e_1, e_2$ as the standard unit-vectors in $\R^2$, and define $a^\perp:=a_1e_2 -a_2e_1\in \R^2$ for any $a=(a_1,a_2)\in \R^2$. Analogously, for functions $g:\Omega\to \R^2$, the map $g^\perp$ is given as $g^\perp(x) = g(x)^\perp$ for all $x\in \Omega$.
For the trace of a matrix $A\in \R^{2\times 2}$, we write $\tr A$, while ${\rm id}:\R^2\to\R^2$ denotes the identity map and  $\mathbb{I}$ its differential.  Futher, $\mathbbm{1}_U:\R^2\to \{0,1\}$ is the indicator function for a set $U\subset \R^2$, i.e., $\mathbbm{1}_U(x)=1$ if and only if $x\in U$.  \CCC Our notation for the dual of a vector space $\Ycal$ is $\Ycal'$, and the corresponding duality pairing is denoted by $\langle \cdot, \cdot\rangle_{\Ycal', \Ycal}$, or simply by $\langle \cdot ,\cdot \rangle$. \BBB
As for $L^p$- and Sobolev spaces, we follow the classical notational conventions, indicating explicitly the target space other than $\R$, we write, e.g.,~$L^2(\Omega;\R^2)$ and $W^{1,2}(\Omega)=W^{1,2}(\Omega;\R)$.Finally, when speaking of the convergence of a family $(a_\eps)_\eps$ with a continuous parameter $\eps>0$, we actually mean that each of the sequences $(a_{\eps_j})_j$ with $\eps_j\searrow 0$ converges for $j\to +\infty$ to the same limit.

\section{Minimizers of the static homogenized limit problem}
\label{sec:min-static}
In this section, we analyze, for every $g\in L^2(\Omega;\R^2)$ and $s\in \Scal^1$, the variational problem
\begin{align}\label{stat_minprob}
\text{minimize}\quad  I\ups_g(u) = E\ups(u)  - \int_\Omega g\cdot u\dd{x}  \qquad \text{for $u\in \Ycal$.} 
\end{align}
Existence of solutions to~\eqref{stat_minprob} follows from the direct method in the calculus of variations. Indeed, observing that $I\ups_g$ is the $\Gamma$-limit resulting from the homogenization procedure studied in~\cite{ChK17}, along with the fact that the term $u\mapsto -\int_{\Omega} g\cdot u\dd{x}$ constitutes a continuous perturbation, yields the lower semicontinuity of the functional $I\ups_g$ regarding the weak topology of $W^{1,2}(\Omega;\R^2)$.  The sequential weak compactness of the sublevel sets of $I\ups_g$ follows from the quadratic growth of \CCC $I_g\ups$  with respect to $\norm{\gamma_u}_{L^2(\Omega)}$ \BBB in combination with the special structure of $\Acal\ups$, which contains only globally rotated shear deformations, see~\eqref{eq:def-A}.

The aim of this section is to address the issues of uniqueness of solutions to~\eqref{stat_minprob} and their explicit characterization. We start by introducing some auxiliary quantities for our analysis: 
For given $\Gamma\in L^2(\Omega)$ and $g\in L^2(\Omega;\R^2)$, let 
\begin{align*}
\Lambda(g, \Gamma) := \widehat{g} + \int_{\Omega} \Gamma g \dd{x} 
\end{align*}
with 
\begin{align}\label{hatg}
\widehat{g} := \int_\Omega x_1 g(x) -x_2 g^{\perp}(x) \dd{x} = \int_{\Omega} \widehat{G}(x)x\dd{x} 
\end{align}
and $\widehat{G}(x): = (g(x)|-g^\perp(x))$ for $x\in \Omega$.

 Given this terminology, one obtains for $u\in \Acal\ups$ with $u=R_u({\rm id}+\Gamma_u e_1)$, cf.~\eqref{eq:A-alt1}, that 
\begin{align}\label{144}
\int_\Omega g\cdot u\dd{x} = \int_\Omega x_1 g\cdot R_ue_1 + x_2 g\cdot (R_ue_1)^\perp + \Gamma_u g\cdot R_ue_1 \dd{x}=  \Lambda(g, \Gamma_u)\cdot R_ue_1. 
\end{align}

Motivated by this observation, we can set up a variational problem that is equivalent to solving~\eqref{stat_minprob} and involves only the shear variable $\Gamma$. This reformulation is made precise in the following lemma.

\begin{lemma}\label{lem:representation}
Let $J_g\ups: W^{1,2}(\Omega)\cap L_0^2(\Omega)\to \R_\infty:=\R\cup\{+\infty\}$ be given by 
 \begin{align*}
 J_g\ups(\Gamma)  & := \begin{cases}\displaystyle \frac{{s_1}^2}{\lambda} \int_{\Omega} ( \partial_2 \Gamma)^2\dd{x} -2s_1s_2\int_{\Omega} \partial_2 \Gamma\dd{x} - |\Lambda(g, \Gamma)| & \text{if $\partial_1 \Gamma=0$, $\partial_2 \Gamma\in K\ups$ a.e.~in $\Omega$,}\\
+\infty & \text{otherwise.} \end{cases}
 \end{align*}
 Then, $u\in \Acal\ups$ is a minimizer of $I_g\ups$ if and only if $\Gamma_u$ is a minimzer of $J_g\ups$ and
 \begin{align}
 \label{eq:char-Gammav}
\begin{cases} R_ue_1 = \overline{\Lambda(g, \Gamma_u)}  & \text{if $\Lambda(g, \Gamma_u)\neq 0$,}\\ R_u\in SO(2) & \text{if $\Lambda(g, \Gamma_u)=0$.} \end{cases}
 \end{align}
Moreover, $\min_{u\in \Ycal} I_g\ups(u)=\min_{\Gamma\in W^{1,2}(\Omega)\cap L_0^2(\Omega)} J_g\ups(\Gamma)$. 
\end{lemma}

\begin{proof}
First, we verify via a simple application of the direct method that minimizers of $J_g\ups$ exist. To check for coercivity, we argue in the case $s\neq e_1$ that the domain of $J_g\ups$ is a bounded subset of $W^{1,2}(\Omega)$ and thus, precompact in the weak topology of $W^{1,2}(\Omega)$. If $s=e_1$,
consider a minimizing sequence $(\Gamma_n)_n$ for $J_g^{(e_1)}$ such that for every $n\in \N$,
$$J_g^{(e_1)}(\Gamma_n)\leq \inf_{\Gamma\in W^{1,2}(\Omega)\cap L_0^2(\Omega)}J_g^{(e_1)}(\Gamma)+\frac1n,$$
and hence, $\partial_1\Gamma_n=0$ and
$$\frac{1}{\lambda} \int_{\Omega} ( \partial_2 \Gamma_n)^2\dd{x} 
\leq - |\Lambda(g, 0)| + \frac{1}{n} +  |\widehat g| + \norm{\Gamma_n}_{L^2(\Omega)}\norm{g}_{L^2(\Omega;\R^2)},$$
so that $(\Gamma_n)_n$ is bounded in $W^{1,2}(\Omega)$ by the Poincar\'e-Wirtinger inequality, and therefore has a weakly converging subsequence in $W^{1,2}(\Omega)$. 
Finally, the existence of a minimizer follows as the functionals $J_g\ups$ are all lower semicontinuous with respect to the weak topology in $W^{1,2}(\Omega)$. 

To prove the statement, we start with the preliminary observation that by~\eqref{144}, 
\begin{equation}
\label{eq:I-J}
I_g\ups(u)=J_g\ups(\Gamma_u)-\Lambda(g,\Gamma_u)\cdot R_u e_1+|\Lambda(g,\Gamma_u)|
\end{equation} 
for every $u\in \Acal\ups$ with $u = R_u({\rm id}+\Gamma_u e_1)$, 
see~\eqref{eq:A-alt1}. 

Now, let $u^\ast$ be a minimizer of $I_g\ups$ in $\Acal\ups$ with $u^\ast = R^\ast({\rm id}+\Gamma^\ast e_1)$. If we consider $u=R({\rm id}+\Gamma^\ast e_1)$ for different $R\in SO(2)$, then~\eqref{eq:char-Gammav} is a direct consequence of the optimality of $u^\ast$ and of \eqref{eq:I-J}. Choosing a competitor $u=R^\ast({\rm id}+\Gamma e_1)$, we infer from the minimality of $u^\ast$ for $I_g\ups$ in combination with \eqref{eq:I-J} that $\Gamma^\ast$ is minimal for $J_g\ups$.

Conversely, let $\Gamma^\ast$ be a minimizer of $J_g\ups$ and let $R^\ast$ be such that \eqref{eq:char-Gammav} holds true. Then, the minimality of the map $u^\ast = R^\ast({\rm id}+\Gamma^\ast e_1)$ for $I_g\ups$ follows once again from \eqref{eq:I-J}.
\end{proof}

\begin{remark}\label{rem:uniqueness}
a) Notice that while the values of the energy functionals $I_g\ups$ depend quadratically and linearly on $\gamma_v$, the dependence of $J_g\ups$ on its domain involves a combination of convex and concave terms. 

b) In the case $s=e_2$, the situation is particularly simple. As the  class of admissible functions does not include any shear deformations and is thus very restrictive, the problem comes down to analyzing rigid-body motion in response to external loading, or in other words, to discussing the standard model of nonlinear rigid elasticity.
To be precise, since $\gamma_u=0$ and $\Gamma_u=0$ for any $u\in \Acal\upee$, and therefore, $\Lambda(g, \Gamma_u)=\widehat{g}$, 
we conclude from Lemma~\ref{lem:representation} that $u$ is a minimizer of $I_g\upee$ if and only if
\begin{align*}
\begin{cases}
R_ue_1=\widehat{g}/|\widehat{g}| & \text{if $\widehat g\neq 0$,}\\
R_u\in SO(2) & \text {otherwise.}
\end{cases}
\end{align*}
\end{remark}

In the case of non-unique rotations for minimizers of $I_g\ups$, we can prove a partial uniqueness result for the shears $\Gamma$. In combination with Lemma~\ref{lem:representation}, it shows that at least one of the building blocks of a minimizer, that is, rotation or shear, is uniquely determined.

\begin{proposition}\label{prop:uniqueness}
Let $u,w\in \Acal\ups$ be two minimizers of $I\ups_g$. If $\Lambda(g, \Gamma_w)=0$, then $\gamma_u=\gamma_w$, or equivalently, $\Gamma_u=\Gamma_w$. 
\end{proposition}

\begin{proof}
For $\delta\in (0,1)$, consider the map $z_\delta:=(1-\delta)u+\delta R_{u}R_{w}^Tw$, which satisfies $z_\delta\in \Acal\ups$ with $R_{z_\delta}=R_u$ and $\gamma_{z_\delta}= (1-\delta)\gamma_u + \delta \gamma_w$, cf.~\eqref{eq:def-A}.  Then,
\begin{align*}
0\leq I\ups_g(z_\delta) - (1-\delta)I\ups_g(u) - \delta I\ups_g(w) 
= (\delta^2-\delta) \frac{s_1^2}{\lambda} \int_{\Omega} (\gamma_u-\gamma_w)^2\dd x - \delta \int_{\Omega} g\cdot (R_uR_w^Tw-w)\dd{x}. 
\end{align*} 
We observe that $R_uR_w^Tw-w= (R_u-R_w)({\rm id}+\Gamma_w e_1)$, so that 
\begin{align*}
\int_\Omega g\cdot (R_uR_w^T-w)\dd{x} & = \Bigl(\int_\Omega x_1g(x) -x_2g^\perp(x) + \Gamma_w(x) g(x)\dd{x}\Bigr)\cdot (R_u-R_w)e_1 \\ & = \Lambda(g, \Gamma_w)\cdot (R_ue_1-R_we_1)=0,
\end{align*}
in view of the assumption $\Lambda(g, \Lambda_w)=0$.
Thus, 
\begin{align}\label{inequ1}
0\geq  \frac{s_1^2}{\lambda} \int_{\Omega} (\gamma_u-\gamma_w)^2\dd x\geq 0, 
\end{align}
which shows that $\gamma_u=\gamma_w$. 
\end{proof}

In what follows, we specify our discussion to the case of a square-shaped reference configuration $\Omega=(-1,1)^2$, which allows for a refined analysis. Indeed, in this setting, one can decouple the two spatial variables and view $\Gamma_u$ and $\gamma_u$ for $u\in \Acal\ups$ as functions of one variable, i.e., as elements of $W^{1,2}(-1,1)$ and $L^2(-1,1)$, respectively. Moreover,
\begin{align*}
\Lambda(g, \Gamma) = \widehat{g}  + \int_{-1}^1 \Gamma(x_2) [g]_{x_1}(x_2) \dd{x_2}\quad \text{and}\quad \widehat{g} = \int_{-1}^1 x_1[g]_{x_2}(x_1)\dd{x_1} - \int_{-1}^1x_2 [g]_{x_1}^\perp(x_2) \dd{x_2}
\end{align*}
with $[g]_{x_1} :=\int_{-1}^1 g(x_1, \cdot)\dd{x_1}$ and $[g]_{x_2}:=\int_{-1}^1 g(\cdot, x_2)\dd{x_2}$. 
 \begin{example}[Examples of \boldmath{$\widehat{g}$ for $\Omega=(-1,1)^2$}]\label{ex:gtilde}
  
 a) If $g$ describes linear loading, i.e., $g(x):=Ax+b$ for $x\in \Omega$ with given (non-trivial) $A\in \R^{2\times 2}$ and $b\in \R^2$, a direct calculation shows that
  $$\widehat{g}= \frac{4}{3} \Big(\begin{array}{c}\tr A \\A_{12}-A_{21}\end{array}\Big).$$ 
If $A$ is symmetric with non-zero trace, then $\widehat{g}/|\widehat{g}|=e_1$, whereas $\widehat{g}/|\widehat{g}| = e_2$ for a skew-symmetric $A$. 
 
 b) When $g$ is constant, then trivially, $\widehat g=0$. In that case, also $\Lambda(g, \Gamma)=0$ for any admissible $\Gamma$ due to the fact that $\Gamma\in L^2_0(-1,1)$ has vanishing mean value. 
Besides, in light of the property that $\Omega$ has its barycenter in the origin, $\widehat{g}=0$ if $g$ is of the form 
\begin{align*}
g(x)=\sum_{i=1}^N A_i  \left(\begin{array}{c} x_1^{\alpha_i}x_2^{\beta_i}\\x_1^{\gamma_i}x_2^{\delta_i}\end{array}\right), \quad x\in \Omega,
\end{align*}
with $A_i\in \R^{2\times 2}$ and $\alpha_i,\beta_i, \gamma_i,\delta_i$ odd integers for $i=1,\dots,N\in \mathbb{N}$. 
 \end{example}

Under special assumptions on the loads or the slip systems, we can identify further conditions on minimizers of $J_g\ups$, and hence also of $I_g\ups$, as the next results illustrate.  
\begin{proposition}\label{prop:g1const} 
Let $\Omega=(-1,1)^2$ and suppose that $[g]_{x_1}$ is constant. If $\Gamma \in W^{1,2}(-1,1)\cap L_0^2(-1,1)$ is a minimizer of $J_g\ups$, then $\Gamma=0$. Every minimizer of  $I_g\ups$ is a rigid-body motion. 
\end{proposition}

\begin{proof}
The case $s=e_2$ is covered in Remark~\ref{rem:uniqueness}. Assume now that $s\neq e_2$. If $[g]_{x_1}$ is constant, then $\Lambda(g, \Gamma)=\widehat g$, and thus, $\Lambda(g, \Gamma)$ is independent of $\Gamma$. The functional $J_g\ups$ is then strictly convex with a unique minimizer $\Gamma$ satisfying 
\begin{align*}
\Gamma'\in {\rm argmin}_{\gamma\in L^2(-1,1;K\ups)} \int_{-1}^1  \frac{s_1^2}{\lambda}\gamma^2 - 2s_1s_2\gamma \dd{x_2}, 
\end{align*}
where $\Gamma'$ denotes the weak derivative of $\Gamma$. 
By Jensen's inequality, this implies that $\Gamma'$ is constant with  
\begin{align}\label{174}
\Gamma'\in {\rm argmin}_{\eta\in K\ups} \frac{s_1^2}{\lambda}\eta^2 -2s_1s_2\eta = \{0\};
\end{align}
for the last identity, we 
make use of the fact that the vertex of the parabola in~\eqref{174}, that is, $\eta=\frac{s_2}{s_1}\lambda$, does not lie in $K\ups$ for $s\neq e_1$.  

Summing up, we have shown that $\Gamma'=0$, and therefore $\Gamma=0$.
The second statement is then an immediate consequence of Lemma~\ref{lem:representation}. 
\end{proof}

For $s=e_1$, we have the following necessary condition for minimizers of $J_g\upe$.

\begin{proposition}\label{prop:necessary_Gamma}
Let $\Omega=(-1,1)^2$, and let $\Gamma\in W^{1,2}(-1,1)\cap L_0^2(-1,1)$ be a critical point of $J\upe_g$. 

a) If $\Lambda(g, \Gamma)= 0$, then $\Gamma=0$.

b) If $\Lambda(g, \Gamma)\neq 0$, the differential equation
\begin{align}\label{phi_prime}
\Gamma''= -\frac{\lambda}{4} \overline{\Lambda(g, \Gamma)}\cdot \Bigl([g]_{x_1} - \frac{1}{2}\int_{-1}^1 [g]_{x_1} \dd{x_2}\Bigr)
\end{align}
holds in the sense of distributions. 
In particular, $\Gamma\in W^{2,2}(-1,1)$.
\end{proposition}

\begin{proof}
We first prove a). Let 
 $\psi\in C_c^\infty(-1,1)$. 
For any $\delta\in \R$, we define the variation 
\begin{align*}
\Gamma_\delta = \Gamma + \delta \Psi,
\end{align*}
where $\Psi$ is the primitive of $\psi$ with vanishing mean value. 
To obtain optimality conditions for the minimizer $\Gamma$, we calculate that 
\begin{align*}
0=\frac{d}{d\delta}_{|_{\delta=0}} J_g\upe(\Gamma_\delta) & = \frac{d}{d\delta}_{|_{\delta=0}} \frac{2}{\lambda} \int_{-1}^1 (\Gamma'+\delta \psi)^2   \dd{x_2} -|\Lambda(g, \Gamma+\delta \Psi)|\\ 
&  = 
 \int_{-1}^1  \frac{4}{\lambda} \Gamma'\psi \dd{x_2} -  \Bigl(\int_{-1}^1[g]_{x_1}\Psi\dd{x_2}\Bigr)  \cdot \overline{\Lambda(g, \Gamma)}=  \int_{-1}^1  \frac{4}{\lambda} \Gamma'\psi \dd{x_2}.
\end{align*}
This implies that $\Gamma'=0$ and concludes the proof in view of the fact that $\Gamma\in L^2_0(-1,1)$.

Next, we prove b).  For any $\Psi\in C^\infty_c(-1,1)$, consider the variation 
\begin{align*}
\Gamma_\delta = \Gamma+\delta \Bigl(\Psi-\frac{1}{2}\int_{-1}^1 \Psi \dd{x_2}\Bigr).
\end{align*}
The claim then follows from the computation
\begin{align*}
0=\frac{d}{d\delta}_{|_{\delta=0}} J_g\upe(\Gamma_\delta) & = \frac{d}{d\delta}_{|_{\delta=0}} \frac{2}{\lambda}  \int_{-1}^1 (\Gamma'+\delta \Psi')^2 \dd{x_2} - \Big|\Lambda\Big(g, \Gamma+\delta \Bigl(\Psi-\frac{1}{2}\int_{-1}^1 \Psi\dd{x_2}\Bigr)\Big)\Big|\\ 
&= 
 \int_{-1}^1  \frac{4}{\lambda} \Gamma'\Psi' \dd{x_2}  -  \Bigl(\int_{-1}^1[g]_{x_1}\cdot \Bigl(\Psi-\frac{1}{2}\int_{-1}^1 \Psi\dd{x_2}\Bigr)\dd{x_2}\Bigr)\cdot \overline{\Lambda(g, \Gamma)} \\ 
 & \displaystyle =  \int_{-1}^1 \frac{4}{\lambda} \Gamma'\Psi -\overline{\Lambda(g, \Gamma)}\cdot \Bigl([g]_{x_1}-\tfrac{1}{2} \int_{-1}^1 [g]_{x_1}\dd{x_2}\Bigr) \Psi\dd{x_2}. 
\end{align*}
\end{proof}

\begin{remark}
Note that case a) can only occur in a scenario where $\widehat g=0$.
If  $[g]_{x_1}$ is constant in addition, one concludes from Proposition~\ref{prop:necessary_Gamma} the necessary condition $\Gamma''=0$ for critical points of $J_g\upe$, which is in agreement with the characterization of minimizers in Proposition~\ref{prop:g1const}.
\end{remark}

\begin{example}[Affine loads] We discuss solutions to \eqref{phi_prime} in the special case of affine loads $g(x) = Ax+b$ for $x\in \Omega=(-1,1)^2$ with $A
\in \R^{2\times 2}$ and $b\in \R^2$. Then, 
\begin{align*}
[g]_{x_1}(x_2) &=2x_2 Ae_2 +2 b \quad \text{for $x_2\in (-1,1)$},\\
[g]_{x_2}(x_1) &=2x_1Ae_1 + 2b \quad \text{for $x_1\in (-1,1)$,}
\end{align*}
so that $[g]_{x_1} - \frac{1}{2} \int_{-1}^1 [g]_{x_1} \dd{x_2}  = 2x_2 Ae_2$, 
$\widehat g = \frac{4}{3}(Ae_1-(Ae_2)^\perp)$ and 
$$\Lambda(g,\Gamma)=\frac{4}{3}(Ae_1-(Ae_2)^\perp)+2Ae_2\int_{-1}^1{x_2\Gamma(x_2)}\dd{x_2}.$$
In view of \eqref{phi_prime}, any critical point for $J_g\upe$ is a polynomial of third order, i.e.,
\begin{center}
$\Gamma(x_2)=\alpha x_2^3+\beta x_2^2+\gamma x_2+\delta$ with coefficients $\alpha, \beta, \gamma, \delta\in \R$. 
\end{center} 
From $\int_{-1}^1 \Gamma(x_2)\dd{x_2}=0$, we infer that $\delta=-\frac{\beta}{3}$.
Since the right-hand side of~\eqref{phi_prime} has null average in $(-1,1)$, we conclude that $\beta=0$ and thus, 
\begin{align*}
\Gamma(x_2)=\alpha x_2^3+\gamma x_2.
\end{align*}
By plugging this structure into the expression of $J_g^{(e_1)}$, the problem of finding minimizers for $J_g^{(e_1)}$  reduces to the following two-dimensional optimization problem:
$$\min_{(\alpha,\gamma)\in \R^2} \frac{4}{\lambda}\Bigl(\frac{9}{5}\alpha^2 +\gamma^2+2\alpha\gamma\Bigr)- \frac{4}{3}\Bigl| Ae_1-(Ae_2)^\perp+Ae_2\Bigl(\frac{3}{5}\alpha+\gamma\Bigr)\Bigr|.$$
\end{example}

\section{Homogenization via evolutionary $\Gamma$-convergence} 
\label{sec:evol}
Before stating our findings on evolutionary $\Gamma$-convergence, we introduce the necessary terminology and recall a few definitions and abstract results from~\cite{MRS08, MiR15}, adjusted to our setting, for a self-contained presentation and the readers' convenience. Let $\Qcal=\Ycal\times \Zcal$, where $\Ycal, \Zcal$ are reflexive, separable 
Banach spaces endowed with the weak topology and $T>0$. We write $q=(y, z)\in \Qcal=\Ycal\times \Zcal$. 
Further, take an energy functional  
$\Ecal: [0, T]\times \Qcal\to \R_\infty:=\R\cup\{\infty\}$ of the form
 \begin{align}\label{Etq}
\Ecal(t, q):=  E(q)  - \CCC \langle l(t), y\rangle_{\Ycal', \Ycal},\BBB
 \end{align} 
where $E:\Qcal\to \R_\infty$ and \CCC $l\in W^{1,1}(0, T; \Ycal')$, \BBB
 and let  
 $\Dcal:\Zcal\times \Zcal\to [0, +\infty]$ be a dissipation distance, i.e., $\Dcal$ is definite and satisfies the triangle inequality.
  A triple $(\Qcal, \Ecal, \Dcal)$ of such a state space, energy and dissipation functional is referred to as an \CCC (energetic) \BBB rate-independent system.

A process $q:[0, T]\to \Qcal$ is called an energetic solution of the rate-independent system $(\Qcal, \Ecal, \Dcal)$
if the global stability condition, i.e., $q(t)\in \Scal(t)$ with 
 \begin{align}\label{globalstability}
\Scal(t) :=\{q\in \Qcal: \Ecal(t, q(t))< +\infty \ \text{and} \ 
\Ecal(t, q)\leq \Ecal(t, \tilde q) +  \Dcal(z, \tilde z) \text{\  \ for all $\tilde q=(\tilde y, \tilde z)\in \Qcal$}\},
 \end{align}
 and the energy balance
\begin{align}\label{energybalance}
\Ecal(t, q(t)) + {\rm Diss}_{\Dcal}(z; [0, t])=  \Ecal(0, q(0)) - \int_0^t \langle \dot{l}(\tau), y(\tau)\rangle \, \dd\tau
\end{align}
hold for all $t\in[0, T]$, see \cite[Definition~2.1.2]{MiR15}; here, $\dot{l}$ denotes the weak derivative of $l$ with respect to time, and for $q=(y,z)\in \Qcal$,
\begin{align*}
{\rm Diss}_{\Dcal}(z; [0, t]):= \sup\Bigl\{\sum_{i=1}^N \Dcal(z(t_{j-1}), z(t_j)): N\in \N, 0=t_0<t_1, \ldots, t_N=t\Bigr\}. 
\end{align*}

\CCC Now, consider \BBB energy and dissipation functionals $\Ecal_\eps$, $\Dcal_\eps$ with parameter $\eps>0$ and $\Ecal_0, \Dcal_0$ as introduced above \CCC (the same notations carry over as well, indicated by subscript $\eps$ and $0$, respectively). \BBB We say that the family $(\Qcal, \Ecal_\eps, \Dcal_\eps)_\eps$ evolutionary $\Gamma$-converges (with well-prepared initial conditions) to $(\Qcal, \Ecal_0, \Dcal_0)$ as $\eps\to 0$, formally 
\begin{align}\label{evGamma}
 (\Qcal, \Ecal_\eps, \Dcal_\eps) \xrightarrow{ev\text{-}\Gamma}
 (\Qcal, \Ecal_0, \Dcal_0)\quad \text{as $\eps\to 0$}, 
\end{align}
if energetic solutions to $(\Qcal, \Ecal_\eps, \Dcal_\eps)$ exist and if the limits of such solutions for $\eps\to 0$ are energetic solutions to $(\Qcal, \Ecal_0, \Dcal_0)$, along with suitable convergences of the energetic and the dissipated contributions; to be precise, if $q_\eps:[0, T] \to \Qcal$ are energetic solutions to $(\Qcal, \Ecal_\eps, \Dcal_\eps)$ such that $q_\eps(t)\to q(t)$ in $\Qcal$ for all $t\in [0,T]$ and $\Ecal_{\eps}(0, q_\eps(0))\to \Ecal_0(0, q(0))$ for $\eps\to 0$, then $q:[0, T]\to \Qcal$ is an energetic solution for $(\Qcal, \Ecal_0, \Dcal_0)$, and it holds that 
\begin{align*}
 \Ecal_\eps(t, q_\eps(t)) \to \Ecal_0(t, q(t)) \quad \text{and}\quad 
{\rm Diss}_{\Dcal_\eps}(z_\eps; [0, t])\to {\rm Diss}_{\Dcal_0}(z; [0, t]) \quad  &\text{for all $t\in [0, T]$,}\\ 
 \langle \dot{l}_\eps(t), y_\eps(t)\rangle \to \langle \dot{l}_0(t), y(t)\rangle  \quad &\text{for a.e.~$t\in [0,T]$.}
\end{align*}

If energetic solutions to $(\Qcal, \Ecal_\eps, \Dcal_\eps)$ do not exist, one needs to employ 
 a refined concept. In fact, we can fall back on the following related approach, based on approximate incremental problems, which automatically admit solutions.   Based on~\cite[Section~4]{MRS08} on the relaxation of evolutionary problems, we state a generalized version suitable for parameter-dependent families of energies and dissipations, as mentioned also in~\cite[Section~2.5.1]{MiR15}.
In the following, we say that $(\Qcal, \Ecal_\eps, \Dcal_\eps)_\eps$ approximately evolutionary $\Gamma$-converges to $(\Qcal, \Ecal_0, \Dcal_0)$ 
and write 
\begin{align}\label{evGammaAIP}
(\Qcal, \Ecal_\eps, \Dcal_\eps)\xrightarrow{ev\text{-}\Gamma_{\text{app}}} 
 (\Qcal, \Ecal_0, \Dcal_0) \quad \text{as $\eps\to 0$,}
\end{align}
\CCC if \BBB sequences of piecewise constant interpolants of approximate solutions to time-incremental problems for $(\Qcal, \Ecal_\eps, \Dcal_\eps)$ have subsequences that converge to energetic solutions of $(\Qcal, \Ecal_0, \Dcal_0)$ in a suitable sense. To be precise: 
For $\eps>0$, let 
\begin{align*}
\Tcal_\eps=\{0=\tau_\eps^{(0)}<\tau_\eps^{(1)}<\ldots <\tau_\eps^{(N_\eps-1)}< \tau_\eps^{(N_\eps)}=T\}
\end{align*} 
with $N_\eps\in \N$, be a family of partitions of $[0, T]$ with fineness $\nu(\Tcal_\eps) = \max_{k=1, \ldots, N_\eps}\tau_\eps^{(k)}-\tau_\eps^{(k-1)}\to 0$ as $\eps\to 0$, and initial conditions $q_\eps^{(0)}\in \Qcal$ with $q_\eps^{(0)}\to q^{(0)}$ in $\Qcal$ and $\Ecal_{\eps}(0, q_\eps^{(0)})\to \Ecal_0(0, q^{(0)})\in \R$ as $\eps\to 0$. 
Consider the piecewise constant functions $q_\eps:[0, T]\to \Qcal$ 
on the partition $\Tcal_\eps$ with 
\begin{align}\label{approxincremental}
q_\eps(t) = \sum_{k=1}^{N_\eps} q_\eps^{(k)} \mathbbm{1}_{[\tau_\eps^{(k-1)}, \tau_{\eps}^k)},
\end{align}
where $q_\eps^{(k)}\in \Qcal$ for $k=1, \ldots, N_\eps$ are iteratively determined solutions to the approximate incremental problem 
\begin{align*}
\Ecal_\eps(\tau_\eps^{(k)}, q_\eps^{(k)}) + \Dcal_\eps(q_{\eps}^{(k-1)}, q_\eps^{(k)}) \leq   \inf_{\tilde q\in \Qcal}\big[\Ecal_\eps(\tau_\eps^{(k)}, \tilde q) + \Dcal_\eps(q_{\eps}^{(k-1)}, \tilde q) \big]+ \nu(\Tcal_\eps)  \delta_\eps^{(k)}; 
\end{align*} 
here, $\delta^{(k)}_\eps>0$ satisfies $\sup_{\eps>0}\sum_{k=1}^{N_\eps}\delta^{(k)}_\eps<+\infty$.
Then,~\eqref{evGammaAIP} means that there exists a subsequence of $(q_\eps)_\eps$ (not relabeled) and an energetic solution $q:[0, T]\to \Qcal$ for $(\Qcal, \Ecal_0, \Dcal_0)$ such that  
\begin{align*}
z_\eps(t)\to z(t) \quad\text{in $\Zcal$}\quad & \text{for all $t\in [0,T]$, } \\ 
\quad \Ecal_\eps(t, q_\eps(t)) \to \Ecal_0(t, q(t)) \quad \text{and}\quad 
{\rm Diss}_{\Dcal_\eps}(z_\eps; [0, t])\to {\rm Diss}_{\Dcal_0}(z; [0, t])\quad &\text{for all $t\in [0,T]$, }\\
 \langle \dot{l}_\eps(t), y_\eps(t)\rangle \to \langle \dot{l}_0(t), y(t)\rangle \quad &\text{for a.e.~$t\in [0,T]$. }
\end{align*}

Next, we collect a list of conditions on $(\Qcal, \Ecal_\eps, \Dcal_\eps)_\eps$ and $(\Qcal, \Ecal_0, \Dcal_0)$, which
 have been shown to provide sufficient criteria for~\eqref{evGamma} and~\eqref{evGammaAIP}: 
\begin{itemize}
\item[(H1)] $\Dcal_\eps=\Dcal$ for  all $\eps>0$, where $\Dcal:\Zcal\times \Zcal\to [0, +\infty]$ is a lower-semicontinuous quasi-distance, i.e., definite and satisfying the triangle inequality; \\[-0.4cm] \CCC
 \item[(H2)] \CCC $E_\eps:\Qcal\to \R_\infty$ are lower semicontinuous for all $\eps>0$;\BBB

\item[(H3)] \CCC $E_\eps(q)\geq C\norm{q}_\Qcal^\alpha -c$ 
 for all $q\in \Qcal$, $\eps>0$ with $C, c>0$ and some $\alpha>1$, and \CCC$l_\eps \to l_0$ in $W^{1,1}(0, T;\Ycal')$; \BBB \\[-0.4cm]
 \item[(H4)] \CCC $(E_\eps)_\eps$ $\Gamma$-converges to $E_0$ with respect to the topology of $\Qcal$; 
 \\[-0.4cm] 
\item[(H5)] \CCC  if $(t_\eps, q_\eps)_\eps$ is a stable sequence, that is, $q_\eps\in \Scal_\eps(t_\eps)$ for all $\eps>0$ and
$ {\rm sup}_{\eps>0}\, \Ecal_\eps(t_\eps, q_\eps)<+\infty$, such that $(t_\eps, q_\eps)\to (t,q)$ in $[0, T]\times \Qcal$, then $q\in \Scal_0(t)$. \BBB
\end{itemize}

The above-listed hypotheses are specializations of the more general assumptions in~\cite[Sections~2.4.2,~2.5.1]{MiR15}, 
taylored to the setting relevant for this work. 
In fact, as a corollary of~\cite[Theorem~2.4.10]{MiR15}, the evolutionary $\Gamma$-convergence~\eqref{evGamma} follows if  
\CCC (H1)--(H5)
hold \CCC and $(\Qcal, \Ecal_\eps, \Dcal_\eps)$ admit energetic solutions.

The proof is based on a by now classical time-discretization strategy. 
If \CCC (H1), (H3), (H4) and (H5) hold, then~\eqref{evGammaAIP} can be considered a consequence of~\cite[Theorem~2.5.1]{MiR15} or \cite[Theorem~4.1]{MRS08}. Strictly speaking, the latter references write the statement explicitly only for relaxation, so assuming that $\Ecal_\eps$ are all identical for all $\eps$, but as mentioned there already, the results carry over to general families of energy functionals after a straightforward modification.  
Both for~\eqref{evGamma} and~\eqref{evGammaAIP}, one obtains that the limit solution $q:[0, T]\to\Qcal$ is measurable.

When it comes to verifying the hypotheses, the critical condition to check is (H5). 
A sufficient condition 
  is the existence of a mutual recovery sequence, cf.~\cite[(2.4.13), Proposition~2.4.8\,(ii), Lemma~2.1.14]{MiR15}:
\begin{itemize}
\item[(H6)] for any sequence $(t_\eps, q_\eps)_\eps \subset [0, T]\times \Qcal$ with
$ {\rm sup}_{\eps>0}\, \Ecal_\eps(t_\eps, q_\eps)<+\infty$ 
that converges to $(t,q)\in [0, T]\times \Qcal$ and any $\tilde q\in \Qcal$, there exists a sequence $(\tilde q_\eps)_\eps\subset \Qcal$ such that $\tilde q_\eps\to q$ in $\Qcal$ and 
\begin{align}\label{general_mutual}
\limsup_{\eps\to 0} \, \Ecal_\eps(t_\eps,\tilde{q}_\eps) + \Dcal_\eps(q_\eps, \tilde{q}_\eps)  - \Ecal_\eps(t_\eps, q_\eps) \leq \Ecal_0(t,\tilde{q}) + \Dcal_0(q,\tilde{q})- \Ecal_0(t, q). 
\end{align}
\end{itemize} 

Now, after this brief excursion into the theory of evolutionary convergence for rate-independent systems, we discuss two different scenarios of such results associated with the 
plasticity model introduced previously in Section~\ref{sec:min-static}. The key part of the proofs is the construction of mutual recovery sequences, where the main difficulty lies in accommodating the non-convex constraints in order to obtain admissible states. A full picture of evolutionary $\Gamma$-convergence in our homogenization setting of crystal plasticity seems currently out of reach. We present in the following a few first steps by studying specific cases where energetic and dissipative effects decouple, see Theorem~\ref{theo:evol2} and Theorem~\ref{theo:evol1} below. 

Let us first introduce the general setting, fix notation and provide some preliminaries.
It is assumed in the remainder of this section that $\Omega=(-1,1)^2$ and $s\in \{e_1, e_2\}$. Let $\Qcal = \Ycal \times \Zcal$ with $\Ycal=W^{1,2}(\Omega;\R^2)\cap L^2_0(\Omega;\R^2)$ and $\Zcal = L^2(\Omega)$, where both spaces are equipped with the corresponding weak topologies. We write $q=(u, \gamma)\in \Qcal$.  

In what follows, we consider dissipation distances $\Dcal:\Qcal\times \Qcal  \to [0, +\infty]$ given for every $q=(u,\gamma),\,\tilde{q}=(\tilde{u},\tilde{\gamma})\in \Qcal$ by either 
\begin{align}\label{D1}
&
\Dcal(q, \tilde q)=\Dcal^1(q, \tilde{q}) := \delta \int_{\Omega} |\gamma-\tilde{\gamma}|\dd{x}  = \delta \norm{\gamma-\tilde \gamma}_{L^1(\Omega)}
\end{align}
with a dissipation coefficient $\delta>0$, or by 
\begin{align}\label{Dgeq}
\Dcal(q, \tilde q) =\Dcal_{\geq}(q, \tilde{q}) := \begin{cases} \Dcal^1(q, \tilde{q}) & \text{if $\gamma\geq \tilde{\gamma}$ a.e.~in $\Omega$,} \\ +\infty & \text{otherwise;}  \end{cases} 
\end{align} 
note that the second choice of dissipation incorporates a monotonicity assumption on the direction of plastic glide, making for a unidirectional process. 
The underlying energy functionals $\Ecal_\eps: [0, T]\times \Qcal \to \R_\infty$ for $\eps> 0$ associated to our system are
 \begin{align}\label{Eeps_Gamma}
 \Ecal_\eps(t, q) := E_\eps(q) - \int_\Omega g_\eps(t)\cdot u \dd x,
 \end{align}
where $g_\eps\in W^{1,1}(0, T; L^2(\Omega;\R^2))$
are given body forces and $E_\eps:\Qcal\to [0,+\infty]$ time-independent energy contributions; the latter will be specified below in each subsection, but share the common property that their admissible states are contained in the sets
\begin{align}\label{Be2eps}
\begin{split}
\Bcal_{\ep}^{(s)} & :=  \{q=(u, \gamma) \in \Qcal: 
\nabla u= R(\Ibb+\gamma s\otimes m), R\in L^\infty(\Omega;SO(2)), \gamma = 0 \text{ on $\eps\Yrig$}\}\\
& = \{q=(u, \gamma_u)\in \Qcal: u\in \Acal^{(s)}_\ep\},
\end{split}
\end{align}
recalling $\Acal_{\eps}^{(s)}$ from~\eqref{eq:def-Aep}, i.e.,
\begin{align*}
\Acal_{\ep}^{(s)} & =  \{u\in \Ycal: 
\nabla u= R_u(\Ibb+\gamma_u s\otimes m), R_u\in L^\infty(\Omega;SO(2)), \gamma_u = 0 \text{ on $\eps\Yrig$}\}. 
\end{align*}
As it was shown in~\cite{ChK17}, the limits of sequences of in $\Bcal^{(s)}_\eps$ can be characterized via 
\begin{align}\label{Bcal_s}
\begin{split}
\Bcal^{(s)} &:=\{q\in \Qcal:  q_\eps\to q \text{ in $\Qcal$},\  q_\eps\in  \Bcal^{(s)}_\eps \text{ for all $\eps>0$} \} \\ 
&= \{(u, \gamma_u)\in \Qcal: u\in \Acal^{(s)}\}, 
\end{split}
\end{align}
with $\Acal^{(s)}$ as defined in~\eqref{eq:def-A}. 

\subsection{The case \boldmath{$s=e_2$}}
For $\eps>0$, consider the energy functional $\Ecal_\eps:[0, T]\times \Qcal\to \R_\infty$ as in~\eqref{Eeps_Gamma} with
\begin{align}\label{Eeps}
E_\eps(q) :=\begin{cases} \displaystyle \int_\Omega \gamma^2\dd{x} & \text{for $q=(u, \gamma) \in \Bcal_\eps^{(e_2)},$}\\ +\infty & \text{otherwise,} 
\end{cases} 
\qquad \text{for $q\in \Qcal,$}
\end{align}
where $\Bcal^{(e_2)}_\eps$ is given in~\eqref{Be2eps}.
The next theorem shows that passing to the limit $\eps\to 0$ in the dissipative system $(\Qcal, \Ecal_\eps, \Dcal)$ yields a purely energetic evolution without any dissipation, see Remark~\ref{rem:no_dissipation}. 
\CCC We point out that 
the energy functionals $E_\eps$ are not lower semicontinuous in light of oscillating rotations, and observe that the oscillations in the shear strains prevent the continuous convergence of the dissipation.
 
As a consequence, 
the existence of energetic solutions to $(\Qcal, \Ecal_\eps, \Dcal)$ is not guaranteed, which suggests to formulate the result within the framework of approximate $\Gamma$-convergence.
Regarding the proof strategy, note
 that classical constructions of mutual recovery sequences such as the ``quadratic trick" (see e.g.~\cite[Section 3.4.5]{MiR15}) cannot be applied here due to the nonlinearity of the problem.  \BBB  

\begin{theorem}[{Evolutionary \boldmath{$\Gamma$}-convergence for \boldmath{$s=e_2$}}]\label{theo:evol2}
Let $\Ecal_\eps$ for $\eps>0$ be as in ~\eqref{Eeps_Gamma} and~\eqref{Eeps}, and let $\Dcal$ be as in~\eqref{D1} or~\eqref{Dgeq}. Further, suppose  \color{black} that $g_\eps \to g_0$ in $W^{1,1}(0, T;L^2(\Omega;\R^2))$ with $g_0\in W^{1,1}(0,T; L^2(\Omega;\R^2))$. \BBB
Then, \begin{align*}
(\Qcal, \Ecal_{\eps}, \Dcal) &\xrightarrow{ev\text{-}\Gamma_{\text{\rm app}}} (\Qcal, \Ecal_0, \Dcal)\quad \text{as $\eps\to 0$,} 
\end{align*}
where 
\begin{align}\label{eq:E0}
\Ecal_0(t, q) :=\begin{cases}\displaystyle  - \int_\Omega g_0(t)\cdot Rx \dd x & \text{for $q=(u,0)\in \Qcal$ with $\nabla u=R$ for $R\in SO(2)$, } \\ +\infty & \text{otherwise.} \end{cases}
\end{align}
\end{theorem}
\begin{proof}
The statement follows, once \CCC (H1), (H3), (H4), (H5) 
\BBB are verified. The first \CCC two \BBB hypotheses are straightforward to check. \CCC Since $E_0$ given by $E_0(q)=0$ for $q=(0, Rx)$ with $R\in \SO(2)$ and $E_0(q)=+\infty$ otherwise 
was characterized as the $\Gamma$-limit of $(E_\eps)_\eps$ \BBB in~\cite{ChK17} (see also~\eqref{eq:defEeps}), (H4) is verified.
For the remaining condition (H5), we construct a suitable mutual recovery sequence according to (H6).

To this end, let $\tilde q\in \Qcal$ and 
 let 
 $(t_\eps, q_\eps)_\eps\subset [0, T]\times\Qcal$ be a sequence of uniformly bounded energy for
 $(\Qcal, \Ecal_\eps, \Dcal)_\eps$  
 converging to $(t, q)\in [0, T]\times \Qcal$. Since $\sup_{\eps>0}\Ecal_\eps(t_\eps, q_\eps)<+\infty$, we have $q_\eps=(u_\eps, \gamma_\eps)\in \Bcal^{(e_2)}_\eps$ for all $\eps$, and therefore, according to~\eqref{Bcal_s},
 \begin{center}
 $q=(u, \gamma)\in \Bcal^{(e_2)}$, \ or equivalently,\quad $u\in \Acal^{(e_2)}$ and $\gamma=\gamma_u$. 
 \end{center} 
This shows that $u$ is affine with $\nabla u=R$ for some $R\in SO(2)$ and that $\gamma_u=0$. Hence, 
$\Ecal_0(t, q)<+\infty$ in view of~\eqref{eq:E0}.
As we aim to prove the estimate~\eqref{general_mutual} for $\Dcal_\eps=\Dcal_0=\Dcal$, there is no loss of generality in assuming $\Ecal_0(t, \tilde q)<+\infty$, i.e., $\tilde q=(\tilde u, 0)\in \Qcal$ with $\tilde u$ affine such that $\nabla \tilde u =\tilde R$ for $\tilde R\in SO(2)$.

We set 
\begin{align*}
\tilde u_\eps=\tilde RR^T u_\eps \in \Acal_\eps^{(e_2)}
\end{align*}
and $\tilde q_\eps=(\tilde u_\eps, \gamma_\eps)\in \Bcal_{\eps}^{(e_2)}$ for any $\eps$. 
Then, $\tilde u_\eps\weakly \tilde{u}$ in $W^{1,2}(\Omega;\R^2)$, and also $\tilde q_\eps\to \tilde q$ in $\Qcal$. 
 Via compact Sobolev embedding, it holds (up to the selection of subsequences) that $\tilde u_\eps \to \tilde u$ in $L^2(\Omega;\R^2)$ and $u_\eps\to u$ in $L^2(\Omega;\R^2)$, which implies
\begin{align*}
\limsup_{\eps\to 0} - \int_\Omega g_\eps(t_\eps)\cdot (\tilde u_\eps - u_\eps)  \dd x =    - \int_\Omega g_0(t)\cdot (\tilde u - u) \dd x.
\end{align*}
In light of $E_\eps(q_\eps) = E_\eps(\tilde q_\eps)$, $\Dcal(q_\eps, \tilde q_\eps)=0$ due to $\tilde \gamma_\eps= \gamma_\eps$, and $\Dcal_0(q, \tilde q)=0$, we therefore obtain
\begin{align*}
\limsup_{\eps\to 0} \Ecal_\eps(t_\eps,\tilde{q}_\eps) - \Ecal_\eps(t_\eps, q_\eps) + \Dcal(q_\eps, \tilde{q}_\eps) = \Ecal_0(t,\tilde{q}) - \Ecal_0(t, q) + \Dcal_0(q,\tilde{q}).
\end{align*}
This yields~\eqref{general_mutual} and shows the existence of a mutual recovery sequence, as desired. 
\end{proof}

\begin{remark}\label{rem:no_dissipation}
Given the system $(\Qcal, \Ecal_0, \Dcal)$, we observe that passing from one admissible state with finite energy to the next does not cause any dissipation. This is because $\Dcal$ is zero when evalutated in finite-energy states and thus, the set of stable states for $(\Qcal, \Ecal_0, \Dcal)$ is the same as that of $(\Qcal, \Ecal_0, 0)$.  
Hence, one may think of the limit system $(\Qcal, \Ecal_0, \Dcal)$ as dissipation-free. 
\end{remark}

We conclude the case $s=e_2$ with a brief discussion of stable states and energetic solutions for the limit system $(\Qcal, \Ecal_0, \Dcal)$. 
Using Remark~\ref{rem:uniqueness}\,b), its set of stable states, denoted for time $t\in [0, T]$ by $\Scal_0(t)$, can be determined via minimization of $\Ecal_0(t, \cdot)$, that is, 
\begin{align*}
\Scal_0(t)  & = \{q\in \Qcal: \Ecal_0(t, q)\leq \Ecal_0(t, \tilde q)\ \text{for all $\tilde q\in \Qcal$}\}  
\\ 
& = \Bigl\{(u,0): u(x)=Rx \text{ for $x\in \Omega$}, R \in  {\rm argmin}_{S\in SO(2)}  -\int_{\Omega} g_0(t)\cdot Sx \dd x 
\Bigr\}\\ 
& =\begin{cases}
\{(Rx, 0): Re_1 = \widehat{g_0(t)}/|\widehat{g_0(t)}| \} & \text{ if $\widehat{g_0(t)}\neq 0$,}\\
\{(Rx, 0): R\in SO(2)\} & \text{ if $\widehat{g_0(t)}=0$};
\end{cases}
\end{align*}
recall~\eqref{hatg} for the definition of $\widehat{g_0(t)}$. In particular, $\Scal_0(t)$ contains only rigid-body motions, and 
the energy balance for $q(t, \cdot) = (u(t, \cdot), 0)\in \Scal_0(t)$ with $u(t)(x)=u(t, x)=R(t)x$ becomes 
\begin{align*}
-R(t)e_1\cdot \widehat{g_0(t)}= -R(0)e_1\cdot \widehat{g_0(0)} - \int_0^t R(\tau)e_1\cdot \widehat{\dot{g_0}(\tau)} \dd{\tau}
\end{align*}
for $t\in [0, T]$, cf.~\eqref{energybalance}.
Under the assumption that $\widehat{g_0(t)}\neq 0$ for all $t\in [0,T]$, we conclude that $(\Qcal, \Ecal_0, \Dcal)$ has a unique energetic solution, which is given by $q(t, \cdot)=(u(t, \cdot), 0)$ with 
\begin{align*}
u(t, x) = \bigl(x_1\widehat{g_0(t)} + x_2\widehat{g_0(t)}^\perp\bigr)/{|\widehat{g_0(t)}|}.
\end{align*}


\subsection{The case \boldmath{$s=e_1$}}

As energy functional $\Ecal_\eps:[0, T]\times \Qcal\to \R_\infty$ for $\eps>0$, we choose either
\begin{align}\label{Eeps1}
\Ecal_\eps(t, q)= \Ecal_\eps^{\rm rig}(t, q) :=\begin{cases} \displaystyle \int_\Omega \gamma^2\dd{x} - \int_{\Omega} g_\eps(t)\cdot u \dd{x} & \text{if $u\in \Acal_\eps^{(e_1), \rm rig},$}\\ +\infty & \text{otherwise,}
\end{cases}
\end{align}
or 
\begin{align}\label{Eeps2}
\Ecal_{\eps}(t,q)=\Ecal_{\eps, \tau_\eps}^{\rm reg}(t, q) 
:= \begin{cases} \displaystyle\int_{\Omega}\gamma^2 \dd{x} -\int_\Omega g_\ep(t)\cdot u\dd{x} + \tau_\eps \norm{\partial_1 \gamma}_{L^2(\Omega)}^2   & \text{if $u\in \Acal_\eps^{(e_1)}$,}\\ 
+\infty,  & \text{otherwise;}
\end{cases}
\end{align}
here, \begin{align*}
\Acal_\eps^{(e_1), \rm rig} := \{u\in \Acal_{\eps}^{(e_1)}: R_u = {\rm const.}\}
\end{align*} 
with $\Acal_\eps^{(e_1)}$ as in~\eqref{eq:def-Aep}, 
$\tau_\eps>0$ are given real constants such that $\tau_\eps\to +\infty$ as $\eps\to 0$, $g_\eps\in W^{1,1}(0, T; L^2(\Omega;\R^2))$ represent loading terms with \color{black} $g_\eps \to g_0$ in $W^{1,1}(0, T;L^2(\Omega;\R^2))$
as $\eps\to 0$.  \color{black}

We remark that, in the first situation with $\Ecal_{\eps}^{\rm rig}$, the inclusion 
\begin{align*} \Acal_\eps^{(e_1), \rm rig}\subset \Acal^{(e_1)} = \{u\in \Ycal: \nabla u=R(\Ibb+\gamma e_1\otimes e_2), R\in SO(2), \gamma\in L^2(\Omega), \partial_1 \gamma=0\}
\end{align*}
for any $\eps>0$, leads to an essentially one-dimensional model already at the level of the $\eps$-dependent functionals.  

The additional term in $\Ecal_{\eps, \tau_\eps}^{\rm reg}$ corresponds to a uni-directional regularization of the shears on the softer layers. As such, it
penalizes oscillations in the $x_1$-direction, but does not affect those in $x_2$, which are relevant to respect the layered structure. 
A common modeling choice for regularizing via higher-order derivatives in models of finite plasticity involves the dislocation tensor $\Gcal$, see e.g.~\cite{MM06}. Considering the multiplicative splitting of the deformation gradient 
into an elastic and a plastic part in~\eqref{multi_decomposition}, $\Gcal$ in the planar case is defined by
\begin{align*}
\Gcal(F_{\rm pl}) := \frac{1}{\det F_{\rm pl}} \curl F_{\rm pl}.
\end{align*}
In the present setting, with slip direction $s=e_1$ and slip plane normal $m=e_2$, it holds that $F_{\rm pl}=\Ibb +\gamma e_1\otimes e_2$ (see~\eqref{Fpl}), and therefore, 
\begin{align*}
\Gcal(F_{\rm pl})= \partial_1 (F_{\rm pl}e_2) - \partial_2 (F_{\rm pl} e_1) = (\partial_1 \gamma) e_1. 
\end{align*} 
Hence, the penalization in $\Ecal_{\eps, \tau_\eps}^{\rm reg}$ can be seen as forcing 
the $L^2$-norm of the dislocation tensor to asymptotically become infinitesimal and leads to a limit model in the context of dislocation-free finite crystal plasticity, as recently studied in~\cite{kruzik.melching.stefanelli, stefanelli}.

We present below two (approximate) evolutionary $\Gamma$-convergence results for the rate-independent systems with the previously introduced energies and dissipations. It is important to notice that both limit systems show no interaction between energy and dissipation terms. This effect is enforced by the rigid energy in $\Ecal_{\ep}^{\rm rig}$ and the regularization in $\Ecal_{\ep,\tau_\ep}^{\rm reg}$, which 
 suppress oscillations in the rotations and shears, respectively. On a technical level, these assumptions are designed to make a limit analysis based on weak-strong convergence arguments feasible. The case of fully oscillating rotations and shears, which requires characterizing limits of products of weakly convergent sequences, may become accessible with new arguments from the theory of compensated compactness (see e.g.~\cite{Mur81, Tar83}), 
  but is beyond the scope of this work and currently still open. 

\begin{theorem}[Evolutionary \boldmath{$\Gamma$}-convergence for \boldmath{$s=e_1$}]
\label{theo:evol1}
With the definitions in~\eqref{Eeps1}, \eqref{Eeps2},~\eqref{D1}, and~\eqref{Dgeq} and 
\begin{align*}
\Ecal_0(t, q) :=\begin{cases}\displaystyle \frac{1}{\lambda}\int_\Omega \gamma^2 \dd{x} - \int_\Omega \color{black} g_0(t) \BBB \cdot u\dd{x}& \text{if $u\in \Acal^{(e_1)},$} \\ +\infty & \text{otherwise,} \end{cases}
\end{align*}
the following evolutionary $\Gamma$-convergence results hold:
\begin{align*}
a)\ (\Qcal, \Ecal^{\rm rig}_{\eps}, \Dcal_{\geq}) & \xrightarrow{ev\text{-}\Gamma}(\Qcal, \Ecal_0, \Dcal_\geq) \ \text{as $\eps\to 0$};\\
b)\  (\Qcal, \Ecal^{\rm reg}_{\eps, \tau_\eps}, \Dcal^1) & \xrightarrow{ev\text{-}\Gamma_{\rm app}} (\Qcal, \Ecal_0, \Dcal^1)  \ \text{as $\eps\to 0$}. 
\end{align*}
\end{theorem}

\begin{proof}
All the basic assumptions of the theory summarized above are satisfied. While \CCC (H1), (H3), as well as (H2) for a), \CCC
 are immediate to check, (H4) \CCC follows from\BBB~\cite{ChK17}. \BBB It remains to prove the existence of mutual recovery sequences, that is, (H6). 

Consider $q, \tilde{q}\in \Qcal$ with $u, \tilde{u}\in \Acal^{(e_1)}$, i.e.,~
\begin{center}
$\nabla u=R(\Ibb + \gamma e_1\otimes e_2)$ \quad and \quad $\nabla \tilde{u} = \tilde{R}(\Ibb+ \tilde{\gamma} e_1\otimes e_2)$ 
\end{center}
with $R, \tilde{R}\in SO(2)$ and $\gamma, \tilde{\gamma}\in L^2(\Omega)$ such that $\partial_1\gamma =\partial_1\tilde{\gamma} = 0$. Moreover, let $(t_\eps, q_\eps)_\eps\subset [0, T]\times \Qcal$ be 
bounded energy sequence for
 $(\Qcal, \Ecal_\eps, \Dcal)_\eps$ with 
 \begin{align}\label{boundedenergy}
\sup_{\eps>0} \Ecal_\eps(t_\eps, q_\eps)<+\infty,
\end{align} 
 such that $(t_\eps, q_\eps)\to (t,q)$ in $[0, T]\times\Qcal$. 

As a consequence of 
 this uniform energy bound, one obtains 
 that $u_\eps\in \Acal^{(e_1)}_\eps$ for all $\eps>0$; in the case of $\Ecal_\ep=\Ecal_\ep^{\rm rig}$, one even has $u_\eps\in \Acal^{(e_1), \rm rig}_{\eps}$. Moreover, we infer from the Radon-Riesz theorem that
\begin{align*}
R_\eps\to R \quad \text{in $L^2(\Omega;\R^{2\times 2})$} \quad \text{and}\quad \gamma_\eps\weakly \gamma \quad \text{in $L^2(\Omega)$,}
\end{align*} 
where, for the sake of a simpler notation, we write $R_\eps$ for $R_{u_\eps}$ and $\gamma_\eps$ for $\gamma_{u_\eps}$. 
The task is to find a sequence $(\tilde{q}_\eps)_\eps\subset \Qcal$ with $\tilde q_\eps\weakly \tilde q$ in $\Qcal$ that satisfies  
\begin{align}\label{recov}
\limsup_{\eps\to 0} \Ecal_\eps(t_\eps,\tilde{q}_\eps) - \Ecal_\eps(t_\eps, q_\eps) + \Dcal(q_\eps, \tilde{q}_\eps) \leq \Ecal_0(t,\tilde{q}) - \Ecal_0(t, 
q)+ \Dcal_0(q,\tilde{q}),
\end{align}
with $\Dcal=\Dcal_0=\Dcal_\geq$ in a) and $\Dcal=\Dcal_0=\Dcal^1$ in b). 
Observe that necessarily, $\tilde{u}_\eps\in \Acal^{(e_1)}_\eps$ for all $\eps>0$, and additionally, for $\Ecal_\ep=\Ecal_\ep^{\rm rig}$, one needs $\tilde{u}_\eps\in \Acal^{(e_1), \rm rig}_\eps$. 
We now detail the construction of mutual recovery sequences for the two scenarios described in a) and b).

a) In this case, the problem is essentially (up to global rotations) one-dimensional with quadratic energy, which allows us to adapt what is frequently referred to as the ``quadratic trick'', cf.~\cite[Section~3.5.4]{MiR15}. \CCC In particular, $(\Qcal, \Ecal_\eps^{\rm rig}, \Dcal_\geq)$ has an energetic solution for each $\eps>0$.
To meet the required constraints for admissible sequences of $\Ecal_\eps^{\rm rig}$, we define $\tilde{u}_\eps\in \Acal_{\eps}^{(e_1)}$ by setting
\begin{align*} 
\tilde{R}_\eps  := \tilde{R}_{u_\eps}= \tilde{R}\qquad  \text{and}\qquad 
\tilde{\gamma}_\eps := \tilde{\gamma}_{u_\eps}= \gamma_\eps + \frac{1}{\lambda} (\tilde{\gamma} - \gamma)\mathbbm{1}_{\eps\Ysoft}
\end{align*}
for all $\eps>0$, recalling that $\lambda\in (0,1)$ denotes the relative thickness of the soft layers. 
Indeed, since $\partial_1 \gamma_\eps=0$ due to $u_\eps \in \Acal^{(e_1)}$, also $\partial_1 \tilde{\gamma}_\eps = 0$, and hence, the vector field $\tilde{R}_\eps(\Ibb+\tilde{\gamma}_\eps e_1\otimes e_2)$ is indeed a gradient field, namely for the potential $\tilde{u}_\eps$. 

In view of $\tilde{\gamma}_\eps\weakly \tilde{\gamma}$ in $L^2(\Omega)$ and $\tilde{u}_\eps\weakly \tilde{u}$ in $W^{1,2}(\Omega;\R^2)$, it follows that
\begin{align}\label{a)energy}\begin{split}
\limsup_{\eps\to 0} \Ecal_\eps^{\rm rig}(t_\eps, \tilde{q}_\eps) - \Ecal_\eps^{\rm rig}(t_\eps, q_\eps) &=\limsup_{\eps\to 0} \int_{\Omega\cap \eps\Ysoft} \frac{1}{\lambda^2} (\tilde{\gamma}-\gamma)^2  + \frac{2}{\lambda}(\tilde{\gamma}-\gamma)\gamma_\eps \dd{x} \\ &\qquad \qquad\qquad\qquad\qquad \qquad-\int_{\Omega} g_\eps(t_\eps)\cdot (\tilde{u}_\eps - u_\eps) \dd{x}  \\ & = \int_\Omega \frac{1}{\lambda} (\tilde{\gamma} - \gamma)^2 +\frac{2}{\lambda}\tilde{\gamma}\gamma - \frac{2}{\lambda}\gamma^2 \dd{x}-\int_\Omega \color{black} g_0(t) \color{black} \cdot (\tilde{u} - u) \dd{x} \\ &=  \int_\Omega \frac{1}{\lambda} \tilde{\gamma}^2 - \frac{1}{\lambda}\gamma^2 \dd{x}-\int_\Omega \color{black}  g_0(t) \color{black}\cdot (\tilde{u} - u)\dd x\\ & =  \Ecal_0(t,\tilde{q}) - \Ecal_0(t, q), 
\end{split}
\end{align}
using that $\mathbbm{1}_{\eps\Ysoft}\weaklystar \lambda\id$ in $L^\infty(\Omega)$ by the Riemann-Lebesgue lemma and that $\mathbbm{1}_{\eps\Ysoft}\gamma_\eps = \gamma_\eps$ for all $\eps$. 

Due to the monotonicity constraint in $\Dcal_{\geq}$, we may assume that ${\gamma}\geq \tilde{\gamma}$, which implies that ${\gamma}_\eps \geq \tilde{\gamma}_\eps$. Then, 
\begin{align}\label{a)diss}
\begin{split}
\limsup_{\eps\to 0} \Dcal_\geq (\gamma_\eps, \tilde{\gamma}_\eps) &= \limsup_{\eps\to 0} \delta \int_\Omega {\gamma}_\eps - \tilde{\gamma}_\eps \dd{x}  = \limsup_{\eps\to 0} \frac{\delta}{\lambda}\int_\Omega \mathbbm{1}_{\eps\Ysoft}(\tilde{\gamma}-\gamma)\dd{x}\\ & =\delta \int_\Omega {\gamma} -\tilde{\gamma} \dd{x} =\Dcal_{\geq}(\gamma, \tilde{\gamma}). 
\end{split}
\end{align}
Combining \eqref{a)energy} and \eqref{a)diss}
gives~\eqref{recov}, as desired. 

b) We start by observing that the uniform energy bound~\eqref{boundedenergy} implies 
\begin{align}\label{lim}
\lim_{\eps\to 0} \int_\Omega |\partial_1 \gamma_\eps|^2 \dd{x} =0;
\end{align}
indeed,
 with $\gamma_\eps=\gamma_{u_\eps}$, we obtain via the Poincar\'e-Wirtinger inequality that
 \begin{align*}
 \int_\Omega \gamma_\eps^2\dd x + \tau_\eps \norm{\partial_1\gamma_\eps}_{L^2(\Omega)}^2 &  \leq  \sup_{\eps>0} \Ecal_\eps(t_\eps, q_\eps)+ \norm{g_\eps}_{W^{1,1}(0,T;L^2(\Omega))} \norm{u_\eps}_{L^2(\Omega)}\\ & \leq C(1+ \norm{\nabla u_\eps}_{L^2(\Omega;\R^{2\times 2})})\leq 4C(1+\norm{\gamma_\eps}_{L^2(\Omega)}),
\end{align*}
with a constant $C>0$ independent of $\eps$. 
This shows that $(\gamma_\eps)_\eps$ is uniformly bounded in $L^2(\Omega)$, as well as 
$\norm{\partial_1\gamma_\eps}_{L^2(\Omega)}\to 0$, considering that $\tau_\eps\to +\infty$ as $\eps$ tends to zero.

Let us define $\tilde{R}_\eps = \tilde{R}$ and
\begin{align}\label{gammahat}
\tilde{\gamma}_\eps =\frac12 \int_{-1}^1\gamma_\eps \dd{x_1} + \frac{1}{\lambda} (\tilde{\gamma} - \gamma)\chi_{\eps\Ysoft}
\end{align}
for $\eps>0$. By construction, we have again that $\partial_1 \tilde{\gamma}_\eps=0$. The ansatz in~\eqref{gammahat} can be viewed as yet a refined version of the modified ``quadratic trick'' in a).

In proving~\eqref{recov}, the convergence of the energy terms follows in analogy to~\eqref{a)energy}, if we account for the fact that $\frac{1}{2}\int_{-1}^1 \gamma_\eps \dd{x_1}\weakly \gamma$ in $L^2(\Omega)$ and if we use the estimate 
\begin{align*}
\int_{\Omega} \Bigl(\frac{1}{2}\int_{-1}^1 \gamma_\eps \dd x_1\Bigr)^2 - \gamma_\eps^2 \dd{x} \leq 0
\end{align*}
 by Jensen's inequality. 

Regarding the dissipative terms, we use the one-dimensional Poincar{\'e} inequality with Poincar{\'e} constant $c>0$ to argue that
\begin{align*}
\limsup_{\eps\to 0}\Dcal^1(\gamma_\eps, \tilde{\gamma}_\eps) & \leq \limsup_{\eps\to 0} \delta \int_\Omega \absB{ \gamma_\eps -\frac12\int_{-1}^1\gamma_\eps \dd x_1} \dd{x} + \delta \displaystyle \int_\Omega \frac{1}{\lambda} |\tilde{\gamma}-\gamma| \chi_{\eps\Ysoft}\dd{x}\\   & \leq \delta \limsup_{\eps\to 0}  c \int_\Omega |\partial_1\gamma_\eps| \dd{x} +\delta \displaystyle \int_\Omega  |\tilde{\gamma}-\gamma| \dd{x} \leq c \delta \lim_{\eps\to 0} \norm{\partial_1 \gamma_\eps}_{L^2(\Omega)}
 + \Dcal^1(\gamma, \tilde{\gamma}).\nonumber
 \end{align*}
The proof of~\eqref{recov} follows then by \eqref{lim}. 
\end{proof}

\begin{remark}
Following the proof of Theorem~\ref{theo:evol1}\,a), it is immediate to see that we get the analogous evolutionary $\Gamma$-convergence result if the monotonicity assumption in the dissipation is dropped, i.e., $(\Qcal, \Ecal^{\rm rig}_{\eps}, \Dcal^1) \stackrel{ev\text{-}\Gamma}{\longrightarrow} (\Qcal, \Ecal_0, \Dcal^1)$ as $\eps\to 0$. 
\end{remark}

\section*{acknowledgements}
E.D.~acknowledges support from the Austrian Science Fund (FWF) through the projects F 65, I 4052, V 662, and Y1292, as well as from BMBWF through the OeAD-WTZ project CZ04/2019. C.K. was supported by the Dutch Research Council (NWO) through the project TOP2.17.01. Most of this work was done while C.K. was affiliated with Utrecht University and partially supported by the Westerdijk Fellowship program. Both authors are grateful for the opportunity to collaborate during a stay at MFO Oberwolfach in the scope of the Research in Pairs program.


\end{document}